\documentclass[11pt]{amsart}

\usepackage[english]{babel} % Language is local.

\usepackage[hyperref,x11names]{xcolor}
\usepackage{hyperref}
\hypersetup{%
colorlinks = true,       % Don't put a box around links
breaklinks = true,       % Allow line breaks inside links
urlcolor   = SteelBlue3, % Web addresses
linkcolor  = Firebrick3, % Cross-references to theorems etc
citecolor  = OliveDrab4, % Cites 
}

\usepackage{todonotes}

\usepackage[pagewise]{lineno}\linenumbers

\usepackage[T1]{fontenc} %% Usual

\usepackage[utf8]{inputenc}

\usepackage[english=british]{csquotes}

\usepackage{amsmath} % ,amscd
\usepackage{amsfonts}
\usepackage{mathrsfs}
\usepackage{amssymb}
\usepackage{amsthm}
\usepackage[capitalise]{cleveref}

\usepackage{microtype}
\newtheoremstyle{example}% name
  {}%      Space above, empty = `usual value'  which is \topsep can be seen with \the\topsep
  {}%      Space below
  {}% Body font
  {\parindent}%         Indent amount (empty = no indent, \parindent = para indent)
  {\bfseries}% Thm head font
  {.}%        Punctuation after thm head
  {\newline}% Space after thm head: \newline = linebreak
  {}%         Thm head spec

\newtheorem{theorem}{Theorem}[section]

\newtheorem{proposition}[theorem]{Proposition}
\newtheorem{corollary}{Corollary}[theorem]
\theoremstyle{example}
\newtheorem{example}{Example}[section]
\theoremstyle{definition}
\newtheorem{definition}[theorem]{Definition}
\newtheorem{remark}[corollary]{Remark}

\newcommand*{\classicsets}[1]{\mathbb{#1}} % Classic sets face

\newcommand*{\reals}{\classicsets{R}}

\renewcommand*{\to}{\myapplication{\varrightarrow}}

\usepackage{tikz}

\usepackage{lscape}

\usepackage{paralist}

\usepackage{mathpazo}

\usepackage{mathrsfs}

\newcounter{parnum}[section]
\renewcommand{\theparnum}{\thesection.\arabic{parnum}}
\renewcommand{\paragraph}{\medskip\refstepcounter{parnum}\textbf{\theparnum}~\textbf}

\usepackage{multirow}
\DeclareMathOperator{\rk}{rk}

\DeclareMathOperator{\HF}{HF}

\newcommand{\C}{\mathbb{C}}

\newcommand*{\field}{\Bbbk}           % The fixed field

\usepackage[normalem]{ulem}
\newcommand*{\cnm}[2]{\genfrac{[}{]}{0pt}{}{{\scriptstyle #1}}{{\scriptstyle #2}}}

\def\to{\longrightarrow}

\DeclareMathOperator{\rank}{rank}

\usepackage[normalem]{ulem}

\usepackage{filecontents}

\begin{document}
\nolinenumbers

\setlength{\abovedisplayskip}{5pt} \setlength{\belowdisplayskip}{5pt}
\setlength{\abovedisplayshortskip}{5pt}\setlength{\belowdisplayshortskip}{5pt}

\begin{abstract}
  We give an explicit formula for the Waring rank of every binary binomial form with complex coefficients.
  % The Waring problem for forms has been widely explored because of its applications in several areas.
  We give several examples to illustrate this, and compare the Waring rank and the real Waring rank for binary binomial forms.
   % Waring problem for forms is important and classical in mathematics.ustrate
  % It has been widely investigated because of its wide applications in several areas.
  % In this paper, we consider the Waring problem for binary forms with complex coefficients.
  % Firstly, we give an explicit formula for the Waring rank of any binary binomial and several examples to illustrating it.
  % Secondly, we prove that, up to scalar multiplication, there are exactly $\binom{d-1}{2}$ binary forms of degree $d$ with Waring rank two and multiple of three fixed distinct linear forms.
\end{abstract}

\title[Binomial formula for binary forms]{The Waring rank of binary binomial forms}

\author[Brustenga i Moncusí]{Laura Brustenga i Moncusí}
\address{Universitat Autònoma de Barcelona, 08193 Bellaterra, Barcelona, Spain}
\email{brust@mat.uab.cat}

\author[Masuti]{Shreedevi K. Masuti}
\address{Department of Mathematics, Indian Institute of Technology Dharwad, WALMI Campus, PB Road, Dharwad - 580011, Karnataka, India}
\email{shreedevi@iitdh.ac.in}

\date{\today}
\thanks{SKM is supported by INSPIRE faculty award funded by Department of Science and Technology, Govt. of India. She was supported by INdAM COFOUND Fellowships cofounded by Marie Curie actions, Italy, and also partially by a grant from Infosys Foundation for her research in Genova and Chennai, respectively, during which part of the work is done. \\
% and which also provided the travel support for participating in this school.\\
LBM is supported by DOGC11/05/2018/19006/3 and partially supported by the Mineco grant MTM2016-75980-P.}


%%%%%%%%%%%%%%%%%%%%%%%%%%%%%%%%%%%%%%%%%%%%%%%%%% 
%%%%%%%%%%%%%%%%%%%%%%%%%%%%%%%%%%%%%%%%%%%%%%%%%%

\maketitle

%%%%%%%%%%%%%%%%%%%%%%%%%%%%%%%%%%%%%%%%%%%%%%%%%%
%%%%%%%%%%%%%%%%%%%%%%%%%%%%%%%%%%%%%%%%%%%%%%%%%%

\section{Introduction}
\label{sec:introduction}
This paper concerns symmetric tensor decomposition as a sum of rank one tensors, which is also known as the Waring problem for forms.
This topic has a rich history and recently received a huge interest mainly because of its wide applicability in areas as diverse as algebraic statistics, biology, quantum information theory, signal processing, data mining, machine
learning, see \cite{comon-2008-symmet-tensor,kolda-2009-tensor-decom-applic,Lan12}. 
\medskip

Let $\field[x,y]$ be the standard graded polynomial ring with coefficients in the field \(\field\subseteq \C\).
For \(d\ge 0\), we denote by \(\field[x,y]_d\) the \(\field\)-vector space of forms of degree \(d\) in \(\field[x,y]\).
For every form \(F\in \field[x,y]_d\), there exist linear forms $L_1,\dots,L_r\in \field[x,y]_1$ and scalars \(a_1,\dots,a_r\in \field\) with \(r\le d+1\) such that
\[
  F=a_1L_1^d+\dots+a_rL_r^d
\] 
(see  \cite[Theorem 4.2]{reznick-2012-length-binary-forms}). When \(\field=\C\) (resp. \(\field=\reals\)), the least of such possible numbers $r$ is called the \emph{Waring rank of $F$} (resp. the \emph{real Waring rank of \(F\)}) and we denote it by \(\rk(F)\) (resp. \(\rk_{\reals}(F)\)).

The Waring problem is more interesting (and challenging) for coefficient in number fields, see \cite{reznick-2012-length-binary-forms}.
Because of the direct connection with the real world, there is also a lot of interest in the real Waring rank, see for instance \cite{BCG11}.
Except for  \cref{Prop:RealBinomial,exm:real-vs-complex-1st,exm:real-vs-complex-2nd,exm:real-vs-complex-3rd}, we will consider \(\field=\C\) throughout this paper.
\medskip

In \cite{AH95} ~J.~Alexander and A.~Hirschowitz found the Waring rank of a generic form in any number of variables, which was a longstanding open problem for more than a hundred years.
However, the Waring rank for generic forms does not provide information for the Waring rank of a specific form. 

There has been an intense research on the Waring rank of binary forms which goes back to the work of J.~J.~Sylvester. 
Sylvester \cite{sylvester-1851-essay-forms-sketch-eli-trans-can-forms,sylvester-1851-remarkable-dicovery-forms-and-hyperdeterminants} gave an explicit algorithm to compute the Waring rank of a binary form.
But in practice, it is unfeasible for real applications.
We refer to \cite{reznick-2012-length-binary-forms} for an excellent survey on the Waring problem for binary forms.
\medskip

As an immediate consequence of Sylverster's algorithm one can give an explicit formula for the Waring rank of monomials in $\C[x,y]$.
Moreover, E.~Carlini, M.~V.~Catalisano and A.~V.~Geramita recently gave an explicit formula for the Waring rank of monomials in any number of variables (see~\cite[Proposition 3.1]{CCG12}). 
This motives us to look beyond monomials for binary forms.
\medskip

The main result of this paper is an explicit formula for the Waring rank of binomials in \(\C[x,y]\) (see~\cref{thm:binomial}).
In general, it is difficult to describe the Waring rank of $F_1+\cdots+F_k$ in terms of the Waring ranks of $F_1,\ldots,F_k$ as is evidenced by Strassen's conjecture (see~\cite{carlini-catalisano-2018-e-computability,CCG12,CCO17,Tei15}). 
In particular, our formula for binomials is far from being a trivial generalization of the monomial case.

Our main technique to compute the Waring rank of a binomial form \(F\in\C[x,y]\) is Sylvester's algorithm (see Section~\ref{SubSec:Syl}).
In order to apply this algorithm, we need to find a form of least degree in the apolar ideal \(F^{\perp}\), and check whether it is square-free (see Section~\ref{SubSec:Syl}).
For this, we give a nonzero form $g_1$ in \(F^{\perp}\) and, computing the Hilbert function (see Section \ref{SubSec:HF}) of \(F^{\perp}\) in certain degrees, we are able to conclude that $g_1$ is of least degree in \(F^{\perp}\).
Hence, we avoid to compute the entire apolar ideal \(F^{\perp}\).
The techniques we use are elementary, but to obtain the correct result is not trivial.
\medskip

The paper is organized as follows: In \cref{sec:preliminaries}, we fix some notation and gather some preliminary results needed for our purpose.
In Section \ref{Section:BinomialForm}, we give an explicit formula for the Waring rank of a binomial form. 
We also give several examples illustrating our result.
We finish Section \ref{Section:BinomialForm} by comparing the Waring and the real Waring rank of binary binomials.

\section{Preliminaries}
\label{sec:preliminaries}

Let \(S\), \(T\) denote respectively the standard graded polynomial rings \(\C[x,y]\), \(\C[X,Y]\).
For \(d\ge 0\), denote respectively by \(S_d\), \(T_d\) the \(d\)-th graded component of \(S\), \(T\).

\paragraph{Apolarity theory and Sylvester's algorithm.}
\label{SubSec:Syl}
Consider the apolar action of \(T\) on \(S\), that is consider \(S\) as a \(T\)-module by means of the differentiation action
$$
\begin{array}{ cccc}
  \circ: & T \times S &\longrightarrow & S \\
         & (g , F) & \to & g \circ F = g(\partial_{X}, \partial_{Y})(F).
\end{array}
$$

\begin{definition}\label{def:apolar-ideal}
  Let \(F\) be a form in \(S_d\).
  A form \(g\in T_{d'}\) is called \emph{apolar to \(F\)} when \(g\circ F\in S_{d-d'}\) is the zero form.
  The \emph{apolar ideal to \(F\)}, denoted as $F^\perp$, is the homogeneous ideal in $T$ generated by all the forms apolar to \(F\), or equivalently
  $$
  F^\perp=\{g \in T\ |\ g \circ F = 0 \}\subseteq T.
  $$
\end{definition}

The so called Sylvester's algorithm below, an algorithm to compute Waring rank of binary forms, is a consequence of Sylvester's Theorem developed in \cite{sylvester-1851-essay-forms-sketch-eli-trans-can-forms,sylvester-1851-remarkable-dicovery-forms-and-hyperdeterminants}. 
Modern proofs of Sylvester's Theorem may be found in \cite[Theorem 2.1]{reznick-2012-length-binary-forms} (which is an elementary proof), \cite[Section 5]{kung-rota-1984-invariant-theory-binary-forms}, and with further discussion in \cite{kung-1986-gundelfinger-bin-forms,kung-1987-canonical-bin-forms-even-deg,kung-1990-canonical-bin-forms-theme-sylvester,reznick-1996-homogeneous-poly-to-pde}.
Here we just state the final version of the algorithm, see \cite[Remark 4.16]{carlini-grieve-oeding-4-lectures}, \cite[Algorithm 2]{bernardi-gimigliano-ida-2011-computing-symmetric-rank} or \cite[Section 3]{comas-seiguer-rank-binary}.
Recall that given $F\in S_d$, by the Structure Theorem (see \cite[Theorem 1.44(iv)]{iarrobino-determinantal-loci}), the apolar ideal \(F^{\perp}\) to \(F\) is complete intersection and it can be generated by forms $g_1,$ $g_2\in T$ with \(\deg(g_1)+\deg(g_2)=d+2\).

\begin{theorem}[Sylvester's algorithm]\label{thr:sylvesters-algorithm}
  Let \(F\) be a form in \(S_d\).
 Let $g_1,g_2\in T$ be generators of the apolar ideal $F^{\perp}$ with $\deg(g_1) \leq \deg(g_2)$. Then,
  \[
    \rk(F)=\begin{cases}
      \deg(g_1) & \mbox{ if } g_1 \mbox{ is square-free}\\
     d+2 -\deg(g_1) & \mbox{ otherwise }.
      %\deg(g_2) & \mbox{ otherwise }.
    \end{cases}
  \]
\end{theorem}

\paragraph{Hilbert function.} 
\label{SubSec:HF}
Let $I$ be a homogeneous ideal in $T$.
The \emph{Hilbert function of $T/I$} is an important numerical invariant associated to $T/I$ defined as
\[
  \HF_{T/I}(d)=\dim_\C T_d - \dim_\C I_d,
\]
where $I_d$ denotes the $\C$-vector space of degree \(d\) forms in $I$.

For $F\in S_d$, let $\langle F\rangle$ denote the $T$-submodule of $S$ generated by $F$, that is, the $\C$-vector space generated by $F$ and by the corresponding derivatives of all orders.
It is well-known that $\langle F \rangle$ determines the Hilbert function of $T/F^\perp$ as follows (see \cite{Ger96}):
\begin{equation}
  \label{Eqn:HFViaDual}\tag{\dag}
  \HF_{T/F^\perp}(d)= \dim_\C \langle F \rangle_d
\end{equation}
where $\langle F \rangle_d$ denotes the $\C$-vector space generated by the forms of degree $d$ in $\langle F \rangle$.

\section{Waring rank of binary binomial forms}
\label{Section:BinomialForm}
In this section, we give an explicit formula for the Waring rank of every binary binomial form (\cref{thm:binomial}).
%Recall that the Waring rank of a binomial form does not depend on its coefficients.
We also give several examples. 
\cref{Examples:thm} illustrates the result itself.
\cref{Example:Trinomial} shows that the Waring rank of a trinomial depends on its coefficients.
\cref{exm:real-vs-complex-4th,exm:real-vs-complex-3rd,exm:real-vs-complex-1st,exm:real-vs-complex-2nd} compare the Waring and the real Waring rank.

Recall that, given a nonzero homogeneous ideal $I$ of $T$, the \emph{initial degree} of \(I\) is the least integer \(i>0\) for which \(I_i\) is not zero.
% \newpage

{Observe that every binary binomial form can be {expressed} as $F=x^ry^s(ay^\alpha+bx^\alpha)$, with~$r, s\geq 0$ and $\alpha \geq 1$. 
Moreover, since the Waring rank of $F$ is invariant by a linear change of coordinates, $\rk(F)$ does not depend on the coefficients $a$ and $b$.
%$F=x^ry^s((\sqrt[\alpha]{a}y)^\alpha+(\sqrt[\alpha]{b}x)^\alpha)$
So, the Waring rank of $F$ is determined by the values of $r,s$ and $\alpha$, as Theorem~\ref{thm:binomial} below shows.}
% Hence, in order to compute $\rk(F)$, we may assume that $r \leq s$, $a=\frac{(r+\alpha)!s!}{r!(s+\alpha)!}$ and $b=1$.

\begin{theorem}\label{thm:binomial}
{Let $F=ax^ry^{s+\alpha}+ bx^{r+\alpha}y^s\in S_d$ be a binomial form, with $ab\neq 0$, $0 \leq r \leq s$ and $1\leq\alpha$.}
  Let \(q,j\) be the unique nonnegative integers such that $r =q\alpha+ j$ with \(0\le j < \alpha\) and set \(\delta=r+\alpha-s\).
  Then the {Waring rank of $F$ can be computed using the following table.} 
  \begin{center}
    \begin{tabular}{|l l|c|}
      \hline
      \multicolumn{2}{|c|}{Conditions} & \(\rk(F)\) \\
      \hline
      ~~ \(\delta \leq 0\) & & \(s+1\)\\
      \hline
      \multirow{4}{8em}{~~ \(\delta > 0\)}
                                       & \(j=0\), \(r=s\) and \(\alpha> 1\) & \(s+2\) \\
                                       & \(j=\delta\) & \(s+1\) \\
                                       & \(j>\delta\) & \(r+\alpha+1\) \\
                                       & otherwise & \(r+\alpha-j\) \\
      \hline
    \end{tabular}
  \end{center}
\end{theorem}
 {In order to prove Theorem~\ref{thm:binomial}, we use Sylvester's algorithm. This involves finding a form $g_1$ of least degree in $F^\perp$ and checking whether it is square-free. 
We take advantage of the Hilbert function of $T/F^\perp$ to conclude that the $g_1$ in $F^\perp$ that we find is indeed a form of least degree in $F^\perp$. 
The computation of this form $g_1$ and whether it is square-free depends on $\delta$ and $j$. For this reason, we split the proof in five different cases: one case (Case (1)) for \(\delta\le 0\), and four cases (Case (2.i-iv)) for $\delta >0$ ((2.i) $0\leq j < \lceil \frac{\delta-1}{2}\rceil;$ (2.ii) $\delta$ is odd and $j=\frac{\delta-1}{2};$ (2.iii) $\lceil \frac{\delta-1}{2} \rceil < j <\delta -1$, OR $\delta$ is even and $j=\frac{\delta}{2};$ and (2.iv) $\delta -1 \leq j$).
  For each case, we find $g_1$, which will be square-free or not depending on additional conditions for the exponents of $F$. This is also a reason for getting five different cases in the table of the theorem. Notice that case $\delta\leq 0$ in the table is covered in Case (1) and $\delta> 0$ is covered in Case (2): the first row of $\delta>0$ in the table is in Case (2.i), the second and third rows of $\delta>0$ in the table are in Case (2.iv) and the last row of the table is covered in Cases (2.i), (2.ii), (2.iii) and (2.iv).}

\begin{proof}[Proof of Theorem~\ref{thm:binomial}]
  Since $\rk(F)$ is invariant by a linear change of coordinates, for the sake of simplicity we assume that $a=\frac{(r+\alpha)!s!}{r!(s+\alpha)!}$ and $b=1$. For every pair of integers \(m,n\),
  we set $\cnm{m}{n}$ equal to $\frac{m!}{(m-n)!}$ if $m \geq n \geq 0$ and equal to zero otherwise.

  \noindent \textbf{Case (1):} Suppose \(\delta\le 0\), that is $s \geq r+\alpha$.
    Clearly, $g_1=X^{r+\alpha+1} \in F^\perp$.
    We claim that the initial degree of $F^\perp$ is $r+\alpha+1$.
    For $0 \leq i \leq s, $ we have
   % \begin{eqnarray*}
      %X^i Y^{s-i} \circ F
      %=\textstyle a \cdot \cnm{r}{i} \cdot \cnm{s+\alpha}{s-i} \cdot x^{r-i}y^{\alpha+i}+
      %\cnm{r+\alpha}{i} \cdot \cnm{s}{s-i} \cdot x^{r+\alpha-i}y^{i} .
    %\end{eqnarray*}
{\begin{eqnarray*}
      X^i Y^{s-i} \circ F
      &=& \begin{cases}
        a \cdot \cnm{r}{i} \cdot \cnm{s+\alpha}{s-i} \cdot x^{r-i}y^{\alpha+i}+
        \cnm{r+\alpha}{i} \cdot \cnm{s}{s-i} \cdot x^{r+\alpha-i}y^{i} & \mbox{ if } 0 \leq i \leq r \\
        \cnm{r+\alpha}{i} \cdot \cnm{s}{s-i} \cdot x^{r+\alpha-i}y^{i} & \mbox{ if } r < i \leq r+\alpha \\
        0 & \mbox{ if } r+\alpha < i \leq s.
      \end{cases}
    \end{eqnarray*}  }  
 \noindent   It is easy to see that the set $\{X^i Y^{s-i} \circ F:0 \leq i \leq r+\alpha\} \subseteq \langle F \rangle_{r+\alpha}$ is $\C$-linearly independent.
    Therefore by \eqref{Eqn:HFViaDual},
    $$
    \HF_{T/F^\perp}(r+\alpha)=\dim_\C \langle F \rangle_{r+\alpha} = r+\alpha+1.
    $$
    This implies that the nonzero homogeneous elements of $F^\perp$ have degree at least $r+\alpha+1$.
    Since $g_1 \in F^\perp$, the initial degree of $F^\perp$ is $r+\alpha+1$.

    Therefore $g_1$ is part of a minimal generating set of $F^\perp$ and hence there exists $g_2 \in T_{s+1}$ such that $F^\perp=(g_1,g_2)$. As $s \geq r+\alpha$ and $g_1$ is not square free, by Sylvester's algorithm
    \[
      \rk(F) = s+1.
    \]

\noindent  \textbf{Case (2):} Assume \(\delta>0\), that is $s<r+\alpha$.
    We split this case in four cases.%, each with its own $g_1$, which will be square-free or not depending on more conditions over the exponents of $F$.
    %compute the Hilbert function of $T/F^\perp$ and use this information to compute the element of smallest degree in $F^\perp.$ Since this computation depends on $j,$ we divide the proof in the following four cases:(i) $ 0\leq j < \lceil \frac{\delta-1}{2}\rceil;$ (ii) $\delta$ is odd and $j=\frac{\delta-1}{2};$ (iii) $\lceil \frac{\delta-1}{2} \rceil < j <\delta -1$, OR $\delta$ is even and $j=\frac{\delta}{2};$ and (iv) $\delta -1 \leq j$.
    %\textcolor{red}{Delete ?}Notice that the first row of (2) in the table is covered in Case (i), the second and the third rows of (2) of the table are covered in Case (iv) whereas the last row of the table is covered in Cases (i),(ii),(iii) and (iv).

    \noindent{\bf Case (2.i):~} $ 0\leq j \le \lceil \frac{\delta-1}{2}\rceil-1$
    
    First we show that the initial degree of $F^\perp$ is $s+j+2$.
    For this it suffices to show that $\HF_{T/F^\perp}(s+j+1)=s+j+2$ and that there exists a nonzero form of degree $s+j+2$ in $F^\perp$. \\
{First assume that $s <\alpha-j-2.$ Then $r \leq s,$ implies that $r < \alpha,$ and hence $r=j.$ 
%    Therefore $r+\alpha-j-1=\alpha-1.$
Therefore
    \begin{multline*}
          X^{\alpha-1-i}Y^i \circ F 
          =\begin{cases}
            \cnm{r+\alpha}{\alpha-1-i} \cdot \cnm{s}{i} \cdot x^{r+i+1} y^{s-i} & \mbox{if } 0 \leq i \leq s\\
            \cnm{r}{\alpha-1-i} \cdot \cnm{s+\alpha}{i} \cdot x^{r+1+i-\alpha}y^{s+\alpha-i} & \mbox{if } \alpha-j-2 <i \leq \alpha-1.
          \end{cases}
        \end{multline*} Hence $\langle F\rangle$ contains all the monomials of degree $s+j+1$.
        Therefore by \eqref{Eqn:HFViaDual}, $\HF_{T/F^\perp}(s+j+1)=s+j+2$. \\
        Now let $s \geq \alpha-j-2.$    We have
     % For \(0\le i\le r+\alpha-j-1\), we have
\begin{multline}
\label{Eqn:C1}  X^{r+\alpha-j-1-i}Y^i \circ F =\\
  =\begin{cases}
    \cnm{r+\alpha}{r+\alpha-j-1-i} \cdot \cnm{s}{i} \cdot x^{j+i+1} y^{s-i} & \mbox{if } 0 \leq i \leq \alpha-j-2\\
    \cnm{r}{r+\alpha-j-1-i} \cdot \cnm{s+\alpha}{i} \cdot a \cdot x^{j+1+i-\alpha}y^{s+\alpha-i} +
    \cnm{r+\alpha}{r+\alpha-j-1-i} \cdot \cnm{s}{i} \cdot x^{j+1+i}y^{s-i} & \mbox{if } \alpha-j -2< i \leq s\\
    \cnm{r}{r+\alpha-j-1-i} \cdot \cnm{s+\alpha}{i} \cdot a \cdot x^{j+1+i-\alpha}y^{s+\alpha-i} & \mbox{if } s<i \leq r+\alpha-j-1 .
  \end{cases}
\end{multline}
%Therefore  taking $i=0,1, \ldots, \min\{s,\alpha-j-2\}$, we get
Thus taking $i=0,1,\ldots,\alpha-j-2$ in \eqref{Eqn:C1} we get
$$\{x^{j+1}y^s,x^{j+2}y^{s-1},\ldots,x^{\alpha-1}y^{s-\alpha+j+2}\} \subseteq \langle F\rangle_{s+j+1}.$$ 
 Notice that $q \alpha-j-2 < s$.  Therefore  taking $i= k\alpha,k\alpha+1,\ldots,k\alpha+\alpha-j-2  $ for $1 \leq k \leq q-1$ in \eqref{Eqn:C1}, we get 
\begin{equation}
\label{Eqn:MonoBig}
\{x^{k\alpha+(j+1)}y^{s-k\alpha},x^{k\alpha+(j+2)}y^{s-1-k\alpha},\ldots,x^{k\alpha+\alpha-1}y^{s-\alpha+j+2-k\alpha}\} \subseteq \langle F\rangle_{s+j+1}
\end{equation}
for all $0 \leq k \leq q-1.$
By assumption on $j$, we have $2(j+1) \leq {\delta}$ and hence $ (q+1) \alpha-(j+1) > s$. 
    Therefore taking $i= (q+1) \alpha-(j+1), (q+1) \alpha-j,\ldots, (q+1) \alpha-1=r+\alpha-j+1$ in \eqref{Eqn:C1}, we get
    \[
      \{x^{q\alpha}y^{s+j+1-q\alpha},x^{q\alpha+1}y^{s+j-q\alpha},\ldots,x^{q\alpha+j}y^{s-q\alpha+1}\} \subseteq \langle F\rangle_{s+j+1}.
    \]
Since $ q \alpha-1 <s$, taking $i= k \alpha-(j+1), k \alpha-j,\ldots, k \alpha-1$ for $k= q,q-1,\ldots,1$ in \eqref{Eqn:C1},   we get
   %$i= q \alpha-1 (<s), q \alpha-2, \ldots,\alpha-j, \alpha-j-1$
 \begin{equation}
 \label{Eqn:MonoSmall}
      \{x^{k\alpha}y^{s+j+1-k\alpha},x^{k\alpha+1}y^{s+j-k\alpha},\ldots,x^{k\alpha+j}y^{s-k\alpha+1}\} \subseteq \langle F\rangle_{s+j+1}
     \end{equation}
for all $0 \leq k \leq q.$  From Equations \eqref{Eqn:MonoBig} and \eqref{Eqn:MonoSmall} we conclude that 
\[
\{x^iy^{s+j+1-i}: 0 \leq i \leq r \}\subseteq \langle F\rangle_{s+j+1}.
\]
Hence  taking $i=r-j,r-j+1,\ldots,s-1,s$ in \eqref{Eqn:C1} we get 
\[
\{x^{r+1}y^{s-r+j},x^{r+2}y^{s-r+j-1},\ldots,x^{s+j}y,x^{s+j+1}\} \subseteq \langle F\rangle_{s+j+1}.
\]
  Therefore we conclude that all the monomials of degree $s+j+1$ belong to $\langle F\rangle_{s+j+1}$.  }
%  \begin{align} 
   %  \label{Eqn:Case1}
     % & X^{r+\alpha-j-1-i}Y^i \circ F \\
      %& \textstyle = a\cdot\cnm{r}{r+\alpha-j-1-i} \cdot \cnm{s+\alpha}{i} \cdot x^{j+1+i-\alpha}y^{s+\alpha-i} + \cnm{r+\alpha}{r+\alpha-j-1-i} \cdot \cnm{s}{i} \cdot
        %x^{j+1+i}y^{s-i}. \nonumber
   % \end{align}
   %By \eqref{Eqn:Case1} for $i=0,1, \ldots, \min\{s,\alpha-j-2\}$, we get
    %$$\{x^{j+1}y^s,x^{j+2}y^{s-1},\ldots,x^{j+1+\min\{s,\alpha-j-2\}}y^{s-\min\{s,\alpha-j-2\}}\} \subseteq \langle F\rangle_{s+j+1}.$$
    %\noindent
    %By assumption on $j$, we have $2(j+1) \leq {\delta}$ and hence $ (q+1) \alpha-(j+1) > s$.
    %Therefore taking $i= (q+1) \alpha-(j+1), (q+1) \alpha-j,\ldots, (q+1) \alpha-1$ in \eqref{Eqn:Case1}, we get
    %\[
      %\{x^{q\alpha}y^{s+j+1-q\alpha},x^{q\alpha+1}y^{s+j-q\alpha},\ldots,x^{q\alpha+j}y^{s-q\alpha+1}\} \subseteq \langle F\rangle_{s+j+1}.
    %\]
  Therefore by \eqref{Eqn:HFViaDual},
    $\HF_{T/F^\perp}(s+j+1)=s+j+2$. Hence the nonzero homogeneous elements of $F^\perp $ have degree at least $s+j+2$.

   \noindent We claim that
    \[
      g_1=\sum_{i=0}^{q} (-1)^i X^{r+1- i\alpha}Y^{i \alpha+s-r+j+1} \in (F^\perp)_{s+j+2}.
    \]
    If $q=0$, then $r=j$. Therefore $g_1=X^{r+1}Y^{s+1}$ which clearly belongs to $F^\perp$.
    Hence assume that $q >0$.
   Then  
 \begin{align*}
      g_1\circ F
      & =0\,+ \cnm{r+\alpha}{r+1}\cdot \cnm{s}{s-r+j+1} \cdot x^{\alpha-1}y^{r-j-1}+\\
      &\quad +\sum_{i=1}^{q-1}(-1)^i \Biggl( a\cdot \cnm{r}{r+1-i\alpha}\cdot \cnm{s+\alpha}{i\alpha+s-r+j+1}\cdot x^{i\alpha-1}y^{(1-i)\alpha+r-j-1}+ \\
      &\qquad\qquad\qquad\ +\cnm{r+\alpha}{r+1-i\alpha} \cdot
        \cnm{s}{i\alpha+s-r+j+1} \cdot x^{(i+1)\alpha-1}y^{-i\alpha+r-j-1} \Biggr)\\
      &\quad +(-1)^q a \cdot \cnm{r}{r+1-q\alpha} \cdot \cnm{s+\alpha}{q\alpha+s-r+j+1} \cdot x^{q\alpha-1}y^{(1-q)\alpha+r-j-1}+0.
 \end{align*}
{Observe that, since $a=\frac{(r+\alpha)!s!}{r!(s+\alpha)!}$, for $i=1,\dots,q$,
 $$a\cdot \cnm{r}{r+1-i\alpha}\cdot \cnm{s+\alpha}{i\alpha+s-r+j+1}=\cnm{r+\alpha}{r+1-(i-1)\alpha} \cdot
 \cnm{s}{(i-1)\alpha+s-r+j+1}.$$
 Hence, the previous sum is telescopic and $g_1\circ F=0$.}
    Thus the initial degree of $F^\perp$ is $s+j+2$. Therefore $F^\perp=(g_1,g_2)$ for some $0 \neq g_2 \in T_{r+\alpha-j}$. Moreover,
    \begin{eqnarray*}
      g_1 = \begin{cases}
        Y^2 \left( \sum_{i=0}^{q} (-1)^i X^{r+1- i\alpha}Y^{i \alpha+s-r+j-1} \right) & \mbox{if } j \geq 1 \mbox{ OR } s -r >0\\
        Y \left( \sum_{i=0}^{q} (-1)^i X^{r+1- i\alpha}Y^{i \alpha} \right) & \mbox{if } j=0 \mbox{ and } r=s.
      \end{cases}
    \end{eqnarray*}
    Suppose that $j=0$ and $r=s$.
    Then
    $$
    \left( \sum_{i=0}^{q} (-1)^i X^{r+1- i\alpha}Y^{i \alpha} \right) (X^\alpha+Y^\alpha)=X^{r+\alpha+1}+(-1)^{q}X Y^{r+\alpha}.
    $$
    Since $X^{r+\alpha+1}+(-1)^{q} X Y^{r+\alpha}$ is square-free, $g_1$ is square-free if $j=0$ and $r=s$. Clearly, if $j \geq 1$, OR $s-r >0$, then
    $g_1$ is not square-free. Therefore by Sylvester's algorithm
    \begin{eqnarray*}
      \rk(F) = \begin{cases}
        r+\alpha-j & \mbox{if } j \geq 1 \mbox{ OR } s-r >0\\
        s+2 & \mbox{if } j=0 \mbox{ and } r=s.\\
      \end{cases}
    \end{eqnarray*}

    \noindent{\bf Case (2.ii):~} $\delta$ is odd and $j=\frac{\delta-1}{2}$

    First we prove that the initial degree of $F^\perp$ is $s+j+1$. For $0 \leq i \leq r+\alpha-j$,
    we have
    %\begin{align}
      %\label{Eqn:Case2}
      %& X^{r+\alpha-j-i}Y^i \circ F \\
      %&= a\cdot\textstyle \cnm{r}{r+\alpha-j-i} \cdot \cnm{s+\alpha}{i} \cdot x^{j+i-\alpha}y^{s+\alpha-i} +
        %\cnm{r+\alpha}{r+\alpha-j-i} \cdot \cnm{s}{i} \cdot x^{j+i}y^{s-i} .\nonumber
    %\end{align}
{
      \begin{multline}\label{Eqn:Case2}
      X^{r+\alpha-j-i}Y^i \circ F=\\
      =\begin{cases}
        \cnm{r+\alpha}{r+\alpha-j-i} \cdot \cnm{s}{i} \cdot x^{j+i}y^{s-i} & \mbox{if } 0 \leq i \leq \alpha-j-1\\
        \cnm{r}{r+\alpha-j-i} \cdot \cnm{s+\alpha}{i} \cdot a \cdot x^{j+i-\alpha}y^{s+\alpha-i} +
        \cnm{r+\alpha}{r+\alpha-j-i} \cdot \cnm{s}{i} \cdot x^{j+i}y^{s-i} & \mbox{if } \alpha-j -1< i \leq s\\
        \cnm{r}{r+\alpha-j-i} \cdot \cnm{s+\alpha}{i} \cdot a \cdot x^{j+i-\alpha}y^{s+\alpha-i} & \mbox{if } s< i \leq r+\alpha-j
      \end{cases}
\end{multline}
    \noindent
    Substituting $i=1,\ldots,\alpha-j-1$ in \eqref{Eqn:Case2}, we get
    $$\{x^{j+1}y^{s-1},x^{j+2}y^{s-2},\ldots,x^{\alpha-1}y^{s-\alpha+j+1}\} \subseteq \langle F\rangle_{s+j}.$$
     As $j+1=\delta-j=r+\alpha-s-j$, we have $j+s+1-\alpha=r-j=q \alpha$.  
   Therefore taking $i=k\alpha+1,k \alpha+2,\ldots,k\alpha+\alpha-j-1 $ for $1 \leq k \leq q-1$ in \eqref{Eqn:Case2}, we get
    \begin{equation}
    \label{Eqn:MonoBigC2}
 \{x^{k\alpha+j+1}y^{s-1-k\alpha},x^{k\alpha+j+2} y^{s-k\alpha-2},\ldots,x^{k\alpha+\alpha-1}y^{s-(k+1)\alpha+j+1}\} \subseteq \langle F\rangle_{s+j}
  \end{equation}
  for all $0 \leq k \leq q-1.$
   Taking $i=s+1,s+2,\ldots,r+\alpha-j$ in \eqref{Eqn:Case2} we get
    \[
      \{x^{q\alpha}y^{s-q\alpha+j},x^{q\alpha+1}y^{s-q\alpha+j-1},\ldots,x^{q\alpha+j}y^{s-q\alpha}\} \subseteq \langle F \rangle_{s+j}.
    \] 
   Since $ q\alpha\leq s$, substituting $i=k \alpha+\alpha-j,k\alpha+\alpha-j+1,\ldots,k\alpha+\alpha$ for $k=q-1,q-2,\ldots,1,0$ in \eqref{Eqn:Case2} we get 
   %  $i=q\alpha-1(<s),q\alpha-2,\ldots,\alpha-j+1,\alpha-j$ in \eqref{Eqn:Case2} we get
  %  \todo[inline]{this $<s$?}
    \begin{equation}
    \label{Eqn:MonoSmallC2}
     \{x^{k\alpha}y^{s-k\alpha+j},x^{k\alpha+1}y^{s-k\alpha+j-1},\ldots,x^{k\alpha+j}y^{s-k\alpha}\} \subseteq \langle F \rangle_{s+j}.
  \end{equation}     
   for all $0 \leq k \leq q.$ From Equations \eqref{Eqn:MonoBigC2} and \eqref{Eqn:MonoSmallC2}
   we conclude that 
\[
\{x^iy^{s+j-i}: 0 \leq i \leq r \}\subseteq \langle F\rangle_{s+j}.
\]
Hence  taking $i=r-j+1,r-j+2,\ldots,s-1,s$ in \eqref{Eqn:Case2} we get 
\[
\{x^{r+1}y^{s-r+j-1},x^{r+2}y^{s-r+j-2},\ldots,x^{s+j-1}y,x^{s+j}\} \subseteq \langle F\rangle_{s+j}.
\]
 Therefore we conclude that  all the monomials of degree $s+j$ are in $\langle F \rangle_{s+j}$.} Hence by \eqref{Eqn:HFViaDual}, $\HF_{T/F^\perp}(s+j)=s+j+1$.
    Therefore the nonzero homogeneous elements of $F^\perp$ have degree at least $s+j+1$.

    We claim that
    \[
      g_1=\sum_{i=0}^{q+1} (-1)^i X^{s+j+1-i \alpha} Y^{i \alpha} \in (F^\perp)_{s+j+1}.
    \]
    Indeed,   
     \begin{align*}
      g_1\circ F
      & =0\,+\cnm{r+\alpha}{s+j+1} \cdot \cnm{s}{0} \cdot x^{r+\alpha-s-j-1}y^s \\
      &\quad + \sum_{i=1}^q (-1)^i \biggl(a\cdot \cnm{r}{s+j+1-i \alpha} \cdot \cnm{s+\alpha}{i \alpha} \cdot x^{r+i\alpha-s-j-1}y^{s+(1-i)\alpha}+\\
      &\quad\qquad\qquad\qquad+\cnm{r+\alpha}{s+j+1-i \alpha} \cdot \cnm{s}{i \alpha} \cdot x^{r+(i+1)\alpha-s-j-1}y^{s-i\alpha} \biggr)\\
      &\quad + (-1)^{q+1} a \cdot \cnm{r}{0} \cdot \cnm{s+\alpha}{(q+1)\alpha} \cdot x^r y^{s-q \alpha}+0.
    \end{align*}
    {Observe that, since $a=\frac{(r+\alpha)!s!}{r!(s+\alpha)!},$ for all $i=1,\dots,q$
\[
  a\cdot \cnm{r}{s+j+1-i \alpha} \cdot \cnm{s+\alpha}{i \alpha} =\cnm{r+\alpha}{s+j+1-(i-1) \alpha} \cdot \cnm{s}{(i-1) \alpha}.
\]  Hence, the previous sum is telescopic and $g_1\circ F=0$.
    }
    Therefore there exists $0\neq g_2\in T_{r+\alpha-j+1}$ such that $F^\perp=(g_1,g_2)$.
    As
    \[
      g_1 \left (X^{\alpha} + Y^\alpha \right) = X^{s+j+1+\alpha}+(-1)^{q+1} Y^{s+j+1+\alpha},
    \]
    and $X^{s+j+1+\alpha}+(-1)^{q+1} Y^{s+j+1+\alpha}$ is square-free, $g_1$ is also square-free. Hence by Sylvester's algorithm $\rk(F)=s+j+1=r+\alpha-j$.

    \noindent{\bf Case (2.iii):~}$ \lceil \frac{\delta-1}{2}\rceil < j < \delta - 1 $ OR $\delta$ is even and $j=\frac{\delta}{2}$

    Let $k=\delta-j$. 
%    Then $ 2\leq k \leq \lceil \frac{\delta-1}{2}\rceil$. 
    We show that $\rk(F)=r+\alpha-j=s+k$. For this it suffices to show that $\HF_{T/F^\perp}(s+k-1)=s+k$ and that there exists a square-free polynomial of degree $s+k$ in $F^\perp$.
    For $0 \leq i \leq r+\alpha-k+1$, we have
  {
    \begin{multline}\label{Eqn:Case3}
       X^{r+\alpha-k+1-i}Y^i \circ F =\\
      =\begin{cases}
        \cnm{r+\alpha}{r+\alpha-k+1-i} \cdot \cnm{s}{i} \cdot x^{k-1+i}y^{s-i} & \mbox{if } 0 \leq i \leq \alpha-k\\
        \cnm{r}{r+\alpha-k+1-i} \cdot \cnm{s+\alpha}{i} \cdot a \cdot x^{k-1+i-\alpha}y^{s+\alpha-i} +
        \cnm{r+\alpha}{r+\alpha-k+1-i} \cdot \cnm{s}{i} \cdot x^{k-1+i}y^{s-i} & \mbox{if } \alpha-k < i \leq s\\
        \cnm{r}{r+\alpha-k+1-i} \cdot \cnm{s+\alpha}{i} \cdot a \cdot x^{k-1+i-\alpha}y^{s+\alpha-i} & \mbox{if } s< i \leq r+\alpha-k+1.
      \end{cases}
    \end{multline}
   %\begin{align}
      %\label{Eqn:Case3}
      %& X^{r+\alpha-k+1-i}Y^i \circ F\\
      %&\textstyle=a\cdot\cnm{r}{r+\alpha-k+1-i} \cdot \cnm{s+\alpha}{i} \cdot x^{k-1+i-\alpha}y^{s+\alpha-i} +
        %\cnm{r+\alpha}{r+\alpha-k+1-i} \cdot \cnm{s}{i} \cdot x^{k-1+i}y^{s-i}.\nonumber
    %\end{align}
    Substituting $i=0,1,\ldots,\alpha-k$ in \eqref{Eqn:Case3} we get
    \[
      \{x^{k-1}y^{s},x^{k}y^{s-1},\ldots,x^{\alpha-1}y^{s-\alpha+k}\} \subseteq \langle F\rangle_{s+k-1}.
    \]
  Notice that $r=q\alpha+j=q\alpha+(\delta-k)$. This implies that $(q+1)\alpha-k=s.$  Hence taking $i= m\alpha,m\alpha+1,\ldots,m\alpha+\alpha-k  $ for $m=1,2,\ldots,q$ in \eqref{Eqn:Case3} we get 
\begin{equation}
\label{Eqn:MonoBigC3}
\{x^{m\alpha+(k-1)}y^{s-m\alpha},x^{m\alpha+k}y^{s-1-m\alpha},\ldots,x^{m\alpha+\alpha-1}y^{s-\alpha+k-m\alpha}\} \subseteq \langle F\rangle_{s+k-1}
\end{equation}
for all $0 \leq m \leq q.$  Now taking $i=(q+1)\alpha-k+1,(q+1)\alpha-k+2,\ldots,(q+1)\alpha-1$ in \eqref{Eqn:Case3} we get
    \[
      \{x^{q\alpha}y^{s-q\alpha+k-1},x^{q\alpha+1}y^{s-q\alpha+k-2},\ldots,x^{q\alpha+k-2}y^{s-q\alpha+1}\} \subseteq \langle F \rangle_{s+k-1}.
    \] 
 %(Note that as $k \leq \lceil \frac{\delta-1}{2}\rceil$, $k \leq \delta-k $ and hence $(q+1)\alpha-1 \leq r+\alpha-k+1$.)\\    
    Now substituting $i=m\alpha-k+1,m\alpha-k+2,\ldots,m\alpha-1$ for $m=q,q-1,\ldots,1$ in \eqref{Eqn:Case3} we get
    %$i=q\alpha-1(\leq s),q\alpha-2,\ldots,\alpha-j+1,\alpha-k+1$ in \eqref{Eqn:Case3} we get
\begin{equation}
    \label{Eqn:MonoSmallC3}
     \{x^{m\alpha}y^{s-m\alpha+k-1},x^{m\alpha+1}y^{s-m\alpha+k-2},\ldots,x^{m\alpha+k-2}y^{s-m\alpha-1}\} \subseteq \langle F \rangle_{s+k-1}
\end{equation}     
   for all $0 \leq m \leq q.$ From Equations \eqref{Eqn:MonoBigC3} and \eqref{Eqn:MonoSmallC3} we conclude that 
   \[
   \{x^{m \alpha}y^{s+k-1-m\alpha},x^{m \alpha+1}y^{s+k-2-m\alpha},\ldots, x^{m \alpha+\alpha-1}y^{s+k-m\alpha-\alpha}:0 \leq m\leq q\} \subseteq \langle F \rangle_{s+k-1}.
   \] 
%Now   taking $i=r-k+2,r-k+3,\ldots,s-1,s$ in \eqref{Eqn:Case3} we get
%\[
   %   \{x^{r+1}y^{s-r+k-2},x^{r+2}y^{s-r+k-3},\ldots,x^{s+k-2}y, x^{s+k-1}\} \subseteq \langle F \rangle_{s+k-1}.
    %\]
  \noindent 
 Notice that $(q+1)\alpha-1=s+k-1.$ Therefore all the monomials of degree $s+k-1$ are in $ \langle F \rangle$, and 
  thus $\langle F \rangle_{s+k-1}=S_{s+k-1}$. }Hence by \eqref{Eqn:HFViaDual}
    $\HF_{T/F^\perp}(s+k-1)=s+k$.

    \noindent Next we claim that
    \[
      g_1=\sum_{i=0}^{q+1} (-1)^iX^{s+k-i\alpha}Y^{i\alpha} \in (F^\perp)_{s+k}.
    \]
    We have
    \begin{align*}
      g_1 \circ F & = 0\, + \cnm{r+\alpha}{s+k} \cdot \cnm{s}{0} \cdot x^{r+\alpha-s-k} y^{s} \\
                  &\quad + \sum_{i=1}^q (-1)^i\biggl(a \cdot \cnm{r}{s+k-i\alpha}\cdot \cnm{s+\alpha}{i\alpha} \cdot x^{r-s-k+i \alpha}y^{s+(1-i)\alpha} +\\
                  &\quad\qquad\qquad\qquad+ \cnm{r+\alpha}{s+k-i\alpha} \cdot \cnm{s}{i\alpha} \cdot x^{r-s-k+(i+1)\alpha}y^{s-i\alpha}\biggr)\\
                  &\quad + (-1)^{q+1} a \cdot \cnm{r}{0} \cdot \cnm{s+\alpha}{(q+1)\alpha} \cdot x^r y^{s-q\alpha}+0.
    \end{align*}
{Observe that, since $a=\frac{(r+\alpha)!s!}{r!(s+\alpha)!}$, for all $i=1,\dots,q$
\[
  a \cdot \cnm{r}{s+k-i\alpha}\cdot \cnm{s+\alpha}{i\alpha} = \cnm{r+\alpha}{s+k-(i-1)\alpha}\cdot \cnm{s}{(i-1)\alpha}.
\] Hence, the previous sum is telescopic and $g_1\circ F=0$.
    }
%    (Notice that $(q+1) \alpha=q \alpha + \alpha=s-\alpha+k+\alpha=s+k>s$). Hence $g_1 \in (F^\perp)_{s+k}$. 
    Also
    \[
      g_1 (X^\alpha+ Y^\alpha)=X^{s+k+\alpha}+(-1)^{q+1} Y^{s+k+\alpha}.
    \]
    As $X^{s+k+\alpha}+(-1)^{q+1} Y^{s+k+\alpha}$ is square-free, $g_1$ is also square-free.
    Therefore by Sylvester's algorithm $\rk(F)=s+k=s+\delta-j=r+\alpha-j$.

    \noindent {\bf Case (2.iv):~} $\delta-1 \leq j \leq \alpha-1$

    First we show that the initial degree of $F^\perp$ is $s+1$. For this it suffices to show that $\HF_{T/F^\perp}(s)=s+1$ and that there exists a nonzero form of degree $s+1$ in $F^\perp$. For $0 \leq i \leq r+\alpha$, we have
 { \begin{multline*}
     X^{r+\alpha-i}Y^i \circ F= \\
      = \begin{cases}
        \cnm{r+\alpha}{r+\alpha-i} \cdot \cnm{s}{i} \cdot x^iy^{s-i} & \mbox{if } 0 \leq i < \alpha \\
        \cnm{r}{r+\alpha-i} \cdot \cnm{s+\alpha}{i} \cdot a \cdot x^{i - \alpha} y^{s+\alpha-i} + \cnm{r+\alpha}{r+\alpha-i} \cdot \cnm{s}{i} \cdot x^{i} y^{s-i} & \mbox{if }\alpha \leq i \leq s\\
        \cnm{r}{r+\alpha-i} \cdot \cnm{s+\alpha}{i} \cdot a \cdot x^{i-\alpha}y^{s+\alpha-i} & \mbox{if } s < i \leq r+\alpha
      \end{cases}
    \end{multline*}
    %\begin{align*}
     % & X^{r+\alpha-i}Y^i \circ F \\
      %&\textstyle =a\cdot\cnm{r}{r+\alpha-i} \cdot \cnm{s+\alpha}{i} \cdot x^{i - \alpha} y^{s+\alpha-i} + \cnm{r+\alpha}{r+\alpha-i} \cdot \cnm{s}{i} \cdot x^{i} y^{s-i}. \nonumber
    %\end{align*}
\noindent Therefore
    \begin{align*}
      & \{x^iy^{s-i}:0 \leq i < \alpha\}\\
      & \bigcup\textstyle\left\{\cnm{r}{r+\alpha-i} \cdot \cnm{s+\alpha}{i} \cdot a\cdot x^{i - \alpha} y^{s+\alpha-i} + \cnm{r+\alpha}{r+\alpha-i} \cdot \cnm{s}{i} \cdot x^{i} y^{s-i}:\alpha \leq i \leq s\right\} \subseteq \langle F \rangle_s.
    \end{align*}
    This implies that $\langle F \rangle_s$ contains all the monomials of degree $s$.} Hence by \eqref{Eqn:HFViaDual},
    \[
      \HF_{T/F^\perp}(s)=s+1.
    \]
    We claim that
    \[
      g_1=\sum_{i=0}^{q+1}(-1)^i X^{s-(j-\delta)-i\alpha}Y^{i\alpha+(j-\delta)+1} \in (F^\perp)_{s+1}.
    \]
     Notice that $(q+1) \alpha = r-j+\alpha=s+\delta-j$. Therefore 
    we have
    \begin{align*}
      g_1\circ F
      & = 0\,+ \cnm{r+\alpha}{s-j+\delta }\cdot \cnm{s}{j-\delta+1} \cdot x^{r+\alpha-s+j-\delta}y^{s-j+\delta-1} \\
      &\quad +\sum_{i=1}^q(-1)^i \biggl(a \cdot \cnm{r}{s-j+\delta-i \alpha}\cdot \cnm{s+\alpha}{i\alpha+j-\delta+1} \cdot x^{r-s+j-\delta+i\alpha}y^{s-j+\delta+(1-i)\alpha-1} \\
      &\qquad\qquad\qquad + \cnm{r+\alpha}{s-j+\delta-i \alpha} \cdot \cnm{s}{i\alpha+j-\delta+1} \cdot x^{r-s+j-\delta+(i+1)\alpha}y^{s-j+\delta-i\alpha-1}\biggr) \\
      & \quad + (-1)^{q+1} a \cdot \cnm{r}{0} \cdot \cnm{s+\alpha}{(q+1)\alpha+(j-\delta+1)} \cdot x^r y^{s-q\alpha-(j-\delta+1)}+0 .
    \end{align*}
{Observe that, since $a=\frac{(r+\alpha)!s!}{r!(s+\alpha)!}$, for all $i=1,\dots,q$,
\[
  a \cdot \cnm{r}{s-j+\delta-i \alpha}\cdot \cnm{s+\alpha}{i\alpha+j-\delta+1} = \cnm{r+\alpha}{s-j+\delta-(i-1) \alpha}\cdot \cnm{s}{(i-1)\alpha+j-\delta+1}.
\] Hence, the previous sum is telescopic and $g_1\circ F=0$.
    }
    Hence $g_1 \in (F^\perp)_{s+1}$.
    Therefore the initial degree of $F^\perp$ is $s+1$. Hence $g_1$ is part of a minimal generating set of $F^\perp$. Therefore there exists $0 \neq g_2 \in T_{r+\alpha+1}$ such that $F^\perp=(g_1,g_2)$. Clearly, if $j \geq \delta+1$, then $g_1$ is not square-free.
    If $\delta-1 \leq j \leq \delta$, then
    \begin{eqnarray*}
      g_1 (X^\alpha+ Y^\alpha) =\begin{cases}
        X^{s+\alpha+1}+(-1)^{q+1} Y^{s+\alpha+1} & \mbox{if }j= \delta-1\\
        Y (X^{s+\alpha}+(-1)^{q+1} Y^{s+\alpha}) & \mbox{if }j=\delta,
      \end{cases}
    \end{eqnarray*}
    which implies that $g_1$ is square-free. Hence by Sylvester's algorithm
    \[
      \rk(F)=\begin{cases}
        s+1 & \mbox{if } \delta-1 \leq j \leq \delta\\
        r+\alpha+1 & \mbox{if } j \geq \delta+1.
      \end{cases}\qedhere
    \]
\end{proof}

\begin{remark}
  \label{Remark:GenSpeRank}
  The generic rank of a form of degree \(r+s+\alpha\) is $\lceil \frac{r+s+\alpha+1}{2} \rceil$, see \cite{AH95} or \cite{carlini-grieve-oeding-4-lectures}.
  \cref{thm:binomial} illustrates that the Waring rank of a specific form behaves as weirdly as possible compared to the generic rank.
  In particular, the Waring rank of a specific form can be smaller or larger than the generic rank.
\end{remark}

We illustrate \cref{thm:binomial} in the following example.

\begin{example}
  \label{Examples:thm}
  Let the notation be as in Theorem \ref{thm:binomial}.
  \begin{asparaenum}[(1)]
  \item [(0)] ($r=s=0$ and $\alpha \geq 1$) Let $F=x^\alpha+y^{\alpha}$. Clearly,
    $$
    \rk(F)=\begin{cases}
      1 & \mbox{if } \alpha=1\\
      2 & \mbox{if } \alpha=2.
    \end{cases}
    $$

    On the other hand, since $r=s=0, $ we have $\delta=\alpha \geq 1$ and $j=0$. Therefore $\rk(F)$ coincides with the Waring rank stated in Theorem \ref{thm:binomial}.
  \item [(1)] ($r=s$ and $\alpha=1$) Let $F=x^ry^{r+1}+x^{r+1}y^r$. In this case $F^\perp=(g_1,g_2)$ where
    \[
      g_1=X^{r+1}-X^rY+\cdots+(-1)^iX^{r+1-i}Y^i+(-1)^{r+1} Y^{r+1} \mbox{ and } g_2=X^{r+2}.
    \]
    Since $g_1$ is square-free, by Sylvester's algorithm $\rk(F)=r+1$. 

    On the other hand, since $r=s$, we have $\delta=r+\alpha-s=\alpha=1$ and hence $j=0$ for every nonnegative integer $r$. Therefore $ \rk(F)=r+1$ by Theorem \ref{thm:binomial} also. 
    
   {We remark that, in \cite[Theorem 3.1]{bruce-tokcan-3-ranks}, the authors proved that if we consider $F \in \field[x,y]$ for different fields $\field\subseteq \C$, then $F$  has at least three different relative Waring ranks. Their proof in fact shows that $\rk(F)=r+1.$}

  \item [(2)] ($r=0,~s>0$ and $\alpha \geq 1$). Let $F=ay^{s+\alpha}+x^{\alpha}y^s$ where $a=\frac{\alpha!s!}{(s+\alpha)!}$. In this case $\delta=r+\alpha-s=\alpha-s$ and $j=0 $. Hence, by Theorem \ref{thm:binomial},
    $$\rk(F)=\begin{cases}
      s+1 & \mbox{if } \alpha \leq s \\
      \alpha & \mbox{if } \alpha >s.
    \end{cases}
    $$ 

    This can be verified directly (without using Theorem \ref{thm:binomial}) as follows: We have 
    \[
      F^{\perp}=\begin{cases}
        (X^{\alpha+1},Y^{s+1} - X^\alpha Y^{s-\alpha+1}) & \mbox{if } \alpha \leq s \\
        (XY^{s+1},X^{\alpha}-Y^{\alpha}) & \mbox{if } \alpha >s
      \end{cases}
    \]
    and hence by Sylvester's algorithm $\rk(F)$ is as required.
  \end{asparaenum}
\end{example}

The following example illustrates that unlike the binomial case, the Waring rank of a trinomial may depend on its coefficients.

\begin{example}
  \label{Example:Trinomial}
  Consider the quadratic form $F=x^2+xy+y^2$. Then
  \[
    F=[x \hspace*{2mm} y]A_F [x \hspace*{2mm} y]^T \mbox{ where }
    A_F=\left(
      \begin{array}{cc}
        1 & 1/2\\
        1/2 & 1
      \end{array}
    \right).
  \] It is a standard fact from linear algebra that $\rk(F)=\rank(A_F)$.
  Hence \(\rk(F)=2\).
  On the other hand, if $G=x^2+2xy+y^2=(x+y)^2$, then $\rk(G)=1$.
  This shows that unlike the binomial case, the Waring rank of a trinomial may depend on its coefficients.
\end{example}

{Now, we compare the Waring and the real Waring rank of a binomial form. For this purpose we recall~\cref{thr:maximal-rank-forms} below due to Sylvester \cite{sylvester-65-Newtons-hitherto-rule} (for a modern proof see \cite[Section 3]{reznick-2012-length-binary-forms}).
\begin{theorem} [Sylvester]\label{thr:maximal-rank-forms}
 Let \(F\) be a form in \(\reals[x,y]_d\) with \(d\ge 3\) which is
 not a \(d\)-th power of a linear form. If $F$ splits into linear factors over $\mathbb{R},$ then $\rk_\mathbb{R}(F) =d.$
 % Then \(\rk(F)=d\) if and only if there are two distinct linear forms \(l_1,l_2\in S_1\) so that \(F=l_1^dl_2\).
\end{theorem}}

First we note the following result which shows that, unlike in the complex case,
there can be at most two possible values of the real rank of a binomial form depending on its coefficients.
\begin{proposition}\label{Prop:RealBinomial}
Consider a real binomial $F=x^ry^s(ay^\alpha+bx^\alpha)$ with $ab \neq 0$.
For $\alpha$ odd, the real Waring rank of $F$ does not depend on the coefficients $a,b$.
For $\alpha$ even, there are at most two different real Waring ranks for $F$, depending on the coefficients $a,b$.
\end{proposition}
\begin{proof}
When $\alpha$ is odd, there are always real $\alpha$-roots of $a$ and $b$ regardless of their sign.
So, via a real linear change of coordinates, we may reduce $F$ to $x^ry^s(y^\alpha+x^\alpha)$, and its Waring rank does not depend on $a$, $b$.
Instead, when $\alpha$ is even, we may reduce $F$ to either $x^ry^s(y^\alpha+x^\alpha)$, if $ab>0$, or $x^ry^s(y^\alpha-x^\alpha)$, if $ab<0$. 
\end{proof}

{
We remark that, for every even $\alpha>0$, there are binomials $F$ for which $x^ry^s(y^\alpha + x^\alpha)$ and $x^ry^s(y^\alpha-x^\alpha)$ have the same real rank. For instance, $\rk_{\mathbb{R}}(x^{2k}+y^{2k})=\rk_{\mathbb{R}}(x^{2k}-y^{2k})=2$ for every $k \geq 1.$ 
In what follows we give explicit examples of real binomial forms with two distinct ranks depending on {their} coefficients (Examples \ref{exm:real-vs-complex-2nd}, \ref{exm:real-vs-complex-4th} and \ref{exm:real-vs-complex-3rd}). First we note the following example which compares the Waring and the real Waring rank of a binomial form. Also note that, since $\alpha$ is odd in the following example, the real Waring rank of $F$ is independent of its coefficients by Proposition \ref{Prop:RealBinomial}.}

%\textcolor{red}{
%From a geometric point of view, the family of binomials in $\P S_d$ is a union of $\binom{\binom{}{}}{}$ projective lines action bla bla bal}

\begin{example}
  \label{exm:real-vs-complex-1st}
  Let $F=x^ry^r(x\pm y)$ where $r \geq 1$.
  By Theorem \ref{thm:binomial} $\rk(F)=r+1$.
  Whereas, since $F$ splits completely into linear factors in $\mathbb{R}$, $\rk_{\mathbb{R}}(F)=2r+1$ by Theorem \ref{thr:maximal-rank-forms}.
\end{example}

%\textcolor{red}{Delete ? Moreover, unlike in the complex case, the real Waring rank of a binomial form may depend on its coefficients.}

\begin{example}
  \label{exm:real-vs-complex-2nd}
  By Theorem \ref{thr:maximal-rank-forms} $\rk_{\mathbb{R}}( x^3-xy^2)=3$ whereas $\rk_{\mathbb{R}}(x^3+xy^2)=2$ by  \cite[Corollary 2.3]{tokcan-2017-warin-rank}. But $\rk(x^3 \pm xy^2)=\rk(y^3 \pm x^2y)=2$ by Theorem \ref{thm:binomial} (take $r=0,~s=1$ and $\alpha=2$ in Theorem \ref{thm:binomial}). 
\end{example}

{The following example is a generalization of Example \ref{exm:real-vs-complex-2nd}. In fact, this example shows that Theorem~\ref{thm:binomial} covers a much larger class of binary forms than binomial, since some forms can be transformed to a binomial form under a suitable linear change of coordinates.
\begin{example}
 \label{exm:real-vs-complex-4th}
Let $F=\ell(x,y)^k(x^2 \pm y^2)$ where $\ell(x,y)=Ax+By$ is a linear form. To begin with, $F$ is not a binomial form in $x$ and $y$, but $F$ can be transformed to a binomial form by a linear change of coordinates. Namely, let $x_1=Bx\mp Ay$ and $y_1=Ax+By$. Then $F(x_1,y_1)=\frac{1}{B^2\pm A^2}y_1^k(x_1^2\pm y_1^2)$, which is a binomial form in $x_1$ and $y_1$. Now by Theorem \ref{thm:binomial} (take $r=0,~s=k$ and $\alpha=2$ in Theorem~\ref{thm:binomial}) $\rk(F)=k+1.$ On the other hand, when $A$ and $B$ are real, we have  $\rk_{\mathbb{R}}(\ell(x,y)^k(x^2 - y^2))=k+2$ by Theorem \ref{thr:maximal-rank-forms} and $\rk_{\mathbb{R}}(\ell(x,y)^k(x^2 + y^2))=k+1$ by \cite[Corollary 2.3]{tokcan-2017-warin-rank}.
\end{example}
}

{
\begin{example}
\label{exm:real-vs-complex-3rd}
Let $F=x^ry^s(x^2-y^2)$. Then by Theorem \ref{thr:maximal-rank-forms} $\rk_\mathbb{R}(F)=r+s+2$. On the other hand, $\rk_\mathbb{R}(x^ry^s(x^2+y^2))< r+s+2$ by \cite[Theorem 2.2]{blekherman-2016-real-rank}.
\end{example}
}

\section*{Acknowledgements}
This research is a Pragmatic project.
We thank the speakers of Pragmatic 2017 for delivering the insightful lectures.
Specially Enrico Carlini, who introduced us this fascinating subject and provided many useful ideas throughout the preparation of this manuscript.
We also thank the organisers of Pragmatic 2017 for providing the local support.
For the support on travel expenses, the second author thanks INdAM. 
We also thank the anonymous referee for his suggestions and comments that has greatly improved the presentation of the paper.

\bibliographystyle{amsplain}
\bibliography{general}

\end{document}